\documentclass{amsart}
\usepackage{latexsym,amsmath,amsthm,amssymb,amsfonts}
\usepackage{wasysym}
\usepackage{setspace}
\usepackage{color}
\usepackage[dvipsnames]{xcolor}
\usepackage{tikz}
\usepackage{url}
\usepackage[colorlinks=true,
            linkcolor=violet,
            urlcolor=Aquamarine,
            citecolor=YellowOrange]{hyperref}
\usetikzlibrary{patterns}

\numberwithin{equation}{section}
\newtheorem{theorem}{Theorem}[section]
\newtheorem{lem}[theorem]{Lemma}
\newtheorem{thm}[theorem]{Theorem}

\newtheorem{cor}[theorem]{Corollary}

\newtheorem{rem}[theorem]{Remark}

\newcommand{\norm}[1]{\left\Vert#1\right\Vert}
\newcommand{\abs}[1]{\left\vert#1\right\vert}
\newcommand{\set}[1]{\left\{#1\right\}}
\newcommand{\Real}{\mathbb R}
\newcommand{\R}{\mathbb R}

\def\s{\,\,\,\,}

\def\dint{\displaystyle{\int}}

\def\endproof{$\hfill\Box$\\}
\def\R{\mathbb{R}}
\def\C{\mathbb{C}}

\def\so{\mathfrak{so}(\R^n)}
\allowdisplaybreaks

\title[3-circle Theorem for Willmore surface I]{\bf 3-circle Theorem for Willmore surface I}

\author{Yuxiang Li}
\address{Department of Mathematical Sciences, Tsinghua University}
\email{liyuxiang@tsinghua.edu.cn}

\author{Hao Yin}
\address{School of Mathematical Sciences, Unviersity of Science and Technology of China}
\email{haoyin@ustc.edu.cn}

\thanks{Y.Li is supported by National Key R\&D Program of China 2022YFA 1005400; H. Yin is supported by NSFC 12141105.}

\date{}
\begin{document}

\begin{abstract}
	In this paper, we study the blow-up of Willmore surfaces. By using the 3-circle theorem, we prove a decay estimate of the second fundamental form along the neck region. This estimate provides a new perspective and streamlined proofs to a few key results in this field, such as the energy identity(quantization), removable singularities and gap theorem.
\end{abstract}

\maketitle
\section{Introduction}
In geometric analysis, especially in the study of variational problems, one often encounters the phenomenon of bubble tree convergence. For example, the harmonic maps from surfaces, the Yang-Mills fields in dimension four, the Willmore surfaces in $\Real^n$ and so on. It is important to understand how various quantities behave after taking this limit.

The study of bubble tree convergence first appeared in the research on harmonic maps from surfaces to manifolds and the study of harmonic heat flow. The bubble tree phenomenon appears because of the $\epsilon$-regularity theorem of the harmonic map equation. In this area, the two most natural questions are: 1. whether the energy identity holds, that is, whether the necks carry away energy; 2. whether the necks exist, and if so, what their structure is. 
Many results have been obtained along this line. For example, when the tension fields are $L^p$-bounded for $p>1$, there are no necks between bubbles \cite{Q-T,D,W-W-Z}; for a harmonic map sequence from a sequence of Riemann surfaces, with diverging complex structures, Chen-Tian proved that the necks converge to geodesics (see \cite{C-T} and \cite{C-L-W}).

In studying neck regions,  typically, we use the Ding-Tian's reduction \cite{D-T} or Parker's reduction \cite{P} to transform the problem into the proof of the following decay estimate:

{\it If $u_k$ is a sequence of harmonic maps defined on $[0,T_k]\times S^1$ with $T_k\rightarrow+\infty$, satisfying 
$$
\lim_{T\rightarrow+\infty}\varlimsup_{k\rightarrow+\infty}\sup_{t\in[T,T_k-T]}\int_{[t,t+1]\times S^1}|\nabla u_k|^2 dtd\theta=0,
$$
then for some $q>0$,
$$
\int_{[t,t+1]\times S^1}|\nabla u_k|^2 dtd\theta\leq C(e^{-qt}+e^{-q(T_k-t)})\qquad \text{for} \quad t\in [0,T_k-1].
$$}
With the above decay estimate, it is quite straightforward to derive
$$
\lim_{T\rightarrow+\infty}\lim_{k\rightarrow+\infty}\int_{[T,T_k-T]\times S^1}|\nabla u_k|^2dtd\theta=0,\s \lim_{T\rightarrow+\infty}\lim_{k\rightarrow+\infty}osc_{[T,T_k-T]\times S^1}u_k=0,
$$
which implies the energy identity and the neckless theorem.

There are two methods to obtain the decay. The first one is to derive various differential inequalities that guarantee the convexity of the tangential energy. This method was originated from Sack-Uhlenbeck \cite{S-U}, and it was implemented independently by Ding \cite{D} and by Lin-Wang \cite{L-W}. The second is the 3-circle theorem introduced by Qing-Tian (see \cite{Q-T}). Note that the 3-circle theorem has a long history and a much wider range of generalizations and applications (for example, \cite{C-D-M,L,X}).  \\


Starting from 2001, Kuwert and Sch\"atzle began using the blow-up analysis methods to study Willmore surfaces and Willmore flows in $\R^n$, obtaining 
$\epsilon$-regularity, gap theorems, and removability theorems for singularities in codimension one \cite{K-S,K-S2,K-S3}. In 2008, Rivi\`{e}re \cite{R} discovered that the Willmore surface equation could be written as a system of two equations with Jacobian structures, allowing the use of Wente's inequality and harmonic analysis methods to study the Willmore equation. Using these equations, Rivi\`ere also obtained $\epsilon$-regularity and removability theorems for singularities in general codimensions.

In this paper, using the 3-circle lemma, we revisit the energy identity and the neck structure problem of a sequence of closed Willmore surfaces with bounded Willmore energy $W(f)$ and bounded genus, with the technical assumption that the conformal structure remains bounded. Using the reduction method due to Ding-Tian \cite{D-T} and Parker \cite{P} again, we can reduce the study of the energy identity and the neck structure to the following problem.

\noindent{\it Let $f_k:[0,T_k]\times S^1\rightarrow\R^n$ be a sequence of conformal and Willmore immersion.
We assume: 

\begin{itemize}

\item[A.1)] $T_k\rightarrow+\infty$.
 
\item[A.2)] for any $t_k\in(0,T_k)$ with $t_k\rightarrow+\infty$ and $T_k-t_k\rightarrow+\infty$, we have 
$$
\int_{[t_k,t_k+1]\times S^1}|A_k|^2dV_{g_k}\rightarrow 0.
$$

\item[A.3)] if $e^{2u_k}(dt^2+d\theta^2)=f_k^*(g_{\Real^n})$, then $\sup_{[0,T_k]\times S^1} \abs{\nabla u_k}\leq \beta$.

\item[A.4)] There is a  Willmore immersion $\tilde{f}_k$ defined on $D$ such that $f_k(t,\theta)= \tilde{f_k}(e^{-t},\theta)$ for all $(t,\theta)\in [0,T_k]\times S^1$.
\end{itemize}

{\bf Q1:} Do we have the energy identity
\begin{equation}\label{Q.energy.identity} 
\lim_{T\rightarrow+\infty}\lim_{k\rightarrow+\infty}
\int_{[T,T_k-T]\times S^1}|H_{k}|^2 dV_{g_k}=0?
\end{equation}

{\bf Q2:} Let ${\bf n}_k$ be the Gauss map (into the Grassmannian $G(2,n)$). Do we have
\begin{equation}\label{Q.no.neck}
\lim_{T\rightarrow+\infty}\lim_{k\rightarrow+\infty}
osc_{[T,T_k-T]\times S^1}{\bf n}_k=0?
\end{equation}

{\bf Q3:} If this is no true, what is the limit of the image of the Gauss map of ${\bf n}_k$?
}

\begin{rem}
	\label{rem:a4}
	The assumption A.4) holds naturally when we consider a convergence of Willmore immersions with the induced conformal structres bounded in the moduli space. 
\end{rem}

\begin{rem}
	\label{rem:a3}	
	The assumption A.3) does not follow from A.1), A.2) and A.4). To see this, simply consider the sequence $f_k(t,\theta)= (e^{-kt}\cos k\theta, e^{-kt}\sin k\theta,0)\in \Real^3$. We do need it for the proof of our main results. However, it arises naturally, when $f_k([0,T_k]\times S^1)$ is a part of closed surface $\Sigma'_k$ with bounded energy and genus. We give a proof of this claim in Lemma \ref{lem:remark}.
\end{rem}

Based on the results in \cite{R}, Bernard and Rivi\`ere, in \cite{B-R}, gave an affirmative answer to Q1 under the assumptions A.1)-A.4). In \cite{M-R2}, Michelat and Rivi\`ere proved some $L^{(2,1)}$ quantization result (Theorem C therein) and provided the affirmative answer to Q2. 

In this paper, we adapt the three circle argument to prove the decay estimate (Theorem \ref{thm:out}), which implies again the affirmative answers to Q1 and Q2.

Before stating our theorem, let's first introduce two definitions. Given $c\in\R^n$, $S\in \so$(the set of $n \times n$ skew-symmetric matrices), and a conformal and Willmore immersion $f:[0,T]\times S^1\rightarrow \R^n$, we define

\begin{eqnarray*}
\tau_1(f,c)&=& \left( -2\int_{\set{t}\times S^1}\partial_t H d\theta - 4\int_{\set{t}\times S^1} (H\cdot A_{ti}) g^{ij} (\partial_j f) d\theta+ \int_{\set{t}\times S^1} \abs{H}^2 \partial_t f  d\theta\right)\cdot c \\
 \tau_2(f,S)&=&2\int_{\set{t}\times S^1} (H\cdot \partial_t (Sf) - \partial_t H\cdot Sf) d\theta - 4\int_{\set{t}\times S^1} (H\cdot A_{ti}) g^{ij} (\partial_j f\cdot Sf) d\theta\\ 
		&& + \int_{\set{t}\times S^1} \abs{H}^2 \partial_t f\cdot Sf d\theta.
\end{eqnarray*}
It will be proved in Section \ref{sec:poho} that $\tau_1$ and $\tau_2$ are independent of $t$ in general and that if $A.4)$ is assumed, then $\tau_1(f,c)=0$ and $\tau_2(f,S)=0$ for any $c\in \Real^n$ and $S\in \mathfrak{so}(\Real^n)$.

We need to point out that $\tau_1(f,c)=0$ corresponds to $d\omega=0$ in the paper by Kuwert-Sch\"atzle \cite{K-S2}, and it is equivalent to the vanishing of the first residue defined in \cite{B-R0}. While the second residue is defined in \cite{B-R0} to be something else, both $\tau_1(f,c)=0$ and $\tau_2(f,S)=0$ were known as conservation laws since \cite{R} and were shown to be related to the Noether's theorem in \cite{B}.

Our main result is:

\begin{thm}
	\label{thm:out}
	For any $q\in (0,2)$ and $m\in \mathbb Z_+$, there exists $\delta_1(m,q)>0$ such that if $f:[0,T]\times S^1\to \Real^n$ is a conformal and Willmore immersion with $f^*(g_{\Real^n})=e^{-2mt+2v}(dt^2+d\theta^2)$, satisfying
	\begin{enumerate}
		\item[(a)] $\tau_1(f,S)=\tau_2(f,c)=0$ for any $S\in \mathfrak{so}(\Real^n),\, c\in \Real^n$;
		\item[(b)] $\sup_{[0,T]\times S^1} \abs{\nabla v}<\delta_1$;
		\item[(c)] $\tilde{\delta}:=\sup_{t\in [0,T-1]}\int_{[t,t+1]\times S^1}\abs{A}^2 dV_g <\delta_1$,
	\end{enumerate}
then
\begin{equation}
	\label{eqn:decayaa}
	\int_{[t,t+1]\times S^1} \abs{A}^2 dV_g \leq C \tilde{\delta} \left( e^{-qt} + e^{-q(T-t)} \right) \qquad \text{for} \quad  t\in [0,T-1].
\end{equation}
\end{thm}
\begin{rem}
We can not expect a decay rate $q>2$ in general and we do not know if $q=2$ is possible.
\end{rem}

\begin{rem}
We refer to Theorem 3.30 in the recent paper \cite{M-R1} of Michelat and Rivi\`ere for a similar but different decay estimate. 
\end{rem}

The proof of this theorem uses the 3-circle argument twice. First, we apply it to the equation of mean curvature $H$,
\[
2\Delta H + 4{\rm div}(H\cdot A_{pq}g^{ip}\partial_i f) -{\rm div}(|H|^2\nabla f)=0.
\]
Under the assumption (c), this equation implies that $H$ is almost a (vector-valued) harmonic function on $[0,T]\times S^1$. In order to apply the 3-circle lemma for harmonic functions, we divide the cylinder $[0,T]\times S^1$ into cylinders of fixed length $Q_i$. We would like to show some convexity property for the quantity $\int_{Q_i} \abs{H}^2 dV_g $ (see Section \ref{sec:3will} for details). This is not true even for an exact harmonic function as indicated by the special case $e^{mt}\cos m\theta$, $e^{mt}\sin m\theta$. It turns out that these two are the only obstructions. Fortunately, thanks to the assumption (a), we are able to show that they do not appear in the limit when we apply the 3-circle to (scaled) $H$. Therefore, we obtain the decay of Willmore energy (Theorem \ref{3-circle.H}).

Next, we apply the 3-circle lemma to the equation of the Gauss map ${\bf n}$,
\[
	\Delta {\bf n} + A_{G(2,{n})}(d{\bf n},d{\bf n})= \nabla^\perp H
\]
where $A_{G(2,n)}$ is the second fundamental form of the Grassmannian $G(2,n)$. Since $A=\nabla {\bf n}$, it suffcies to prove the decay for the Dirichlet energy of ${\bf n}$ along the cylinder. For that purpose, we compare the size of $H$ and $A$. For those parts of the cylinder on which $A$ is smaller than a multiple of $H$, the decay of $A$ follows from that of $H$, which is known. For those parts that $H$ is significantly smaller than $A$, the above equation shows that ${\bf n}$ is ``almost" a harmonic map. By now, it is standard to apply the 3-circle lemma to study the energy decay of a harmonic map from cylinder and a key ingredient of the proof is some Pohozaev identity which implies that the tangential energy on each $\set{t}\times S^1$ is the same as the radial energy. In our case, we are able to show that this is ``almost" true when $H$ is much smaller than $A$ (see Lemma \ref{lem:poho.A} for exact statement). Note that ${\bf n}$ is only defined on a cylinder (that is not simply connected) and in general we do not expect the Pohozaev identity even for harmonic maps defined on cylinders.

The details of these two steps appear in Section \ref{sec:3will} and Section \ref{sec:A} respectively. Together they prove Theorem \ref{thm:out}. In Section \ref{sec:decay}, we apply Theorem \ref{thm:out} to give positive answers to Q1 and Q2 under the assumption of A.1)-A.4). The main task is to verify the assumption (b). This follows from A.2) and A.3) via a proof by contradiction. 

In Section \ref{sec:app}, we discuss further applications of Theorem \ref{thm:out}. It is natural that a decay estimate as \eqref{eqn:decayA} has applications such as removable singularity theorem and gap theorem. A problem is that for these applications, it is not natural to assume A.4), hence we will have to prove the assumption (a) before we can use Theorem \ref{thm:out}. In Section \ref{subsec:remove}, we reprove the removable singularity theorem due to Kuwert and Sch\"atzle \cite{K-S2, K-S3} and Rivi\`ere \cite{R} using our decay estimate. In Section \ref{subsec:gap}, we prove a version of gap theorem for conformal and Willmore immersions defined on $\Real \times S^1$.  \\

To conclude the introduction, we discuss briefly the results on Willmore immersions with the assumptions A.1)-A.3), but not A.4). Under conditions A.1)-A.3), Laurain and Rivi\`ere studied problem Q1 in \cite{L-R} and they provide some sufficient condition for the validity of the energy identity \eqref{Q.energy.identity}. Very recently, Dorian Martino proved that energy identity is true when $n=3$ and the index is bounded \cite{M}. He also proved that the conformal Gauss map converges to a geodesic. In a subsequent paper, we will apply the 3-circle theorem again to show that for a sequence of conformal and Willmore immersions satisfying A.1)-A.3), but not A.4), the image of the Gauss map of $f_k$ converges to a geodesic in $G(2,n)$.

\vskip 1cm
\noindent {\bf Acknowledgement.} The authors would like to thank Professor Ernst Kuwert for his insights and helpful discussions.

\section{A Pohozaev type identity}
\label{sec:poho}

As in the proof of the energy identity and no neck theorem in the harmonic map case (see \cite{D-T,L-W,Q-T} for example), it is important to have some Pohozaev identity. In the case of a harmonic map $u:D\to N$, where $D$ is the unit disk in $\Real^2$ and $N$ is any target manifold, the Pohozaev identity implies that the tangential part of the Dirichet energy is exactly the same as the radial part. In the Willmore case, it is not immediately clear what are the tangential part and the radial part of the energy. Instead, we find that the following result will serve a similar purpose in the proof of energy decay in Theorem \ref{thm:out}.

\begin{thm}
	\label{thm:poho} Let $f:D \to \Real^n$ be a smooth conformal Willmore immersion. Let $(r,\theta)$ be the polar coordinate and $t=-\log r$. For any $t>0$, any $c\in \Real^n$ and any $S\in \mathfrak{so}(\Real^n)$, we have
	\begin{equation}
		\label{eqn:poho}
		\begin{split}
			0=& 2\int_{\set{t}\times S^1} (H\cdot \partial_t (Sf) - \partial_t H\cdot (Sf+c)) d\theta - 4\int_{\set{t}\times S^1} (H\cdot A_{ti}) g^{ij} (\partial_j f\cdot (Sf+c)) d\theta \\
			  & + \int_{\set{t}\times S^1} \abs{H}^2 \partial_t f\cdot (Sf+c) d\theta.
		\end{split}
	\end{equation}
Or equivalently, 
\[
	\tau_1(f,c)=\tau_2(f,S)=0.
\]
Here $g$ is the pullback metric $f^*(g_{\Real^n})$ and $A_{ti}$ is $A(\partial_t f,\partial_i f)$ with $A$ the second fundamental form. 
\end{thm}
\begin{rem}
	This result is not new. For a clear presentation of how such identities can be obtained via Noether's theorem, we refer to Bernard's paper \cite{B}. Here we have put it in a form convenient for us and include a proof for completeness.
\end{rem}

\begin{proof}
Let $f_s:D\rightarrow\R^n (s\in (-\epsilon,\epsilon))$ be a smooth variation with $f_0=f$ and $\dot{f_s}|_{s=0}=\varphi$. 
Denote the pullback metric $f_s^*(g_{\Real^n})$ by $g_s$ and the mean curvature vector $\Delta_{g_s} f_s$ by $H_s$. We omit the subscript $s$ when $s=0$.

For any subdomain $\Omega\subset D$, direct computation shows
\begin{eqnarray*}
&&\frac{d}{ds}W(f_s,\Omega)|_{s=0} \\
&=&\frac{d}{ds}\int_{\Omega}(g_s^{ij}(f_s)_{ij}-\Gamma_{s,ij}^k\partial_k f_s)^2\sqrt{|g_s|}dx|_{s=0}\\
&=&2\int_{\Omega}H\cdot \Delta_g\varphi\sqrt{\abs{g}}dx
+2\int_{\Omega}H\cdot (\dot{g}_{s}^{ij}\partial^2_{ij}f-\dot{\Gamma}_{s,ij}^k\partial_k f)|_{s=0} \sqrt{\abs{g}} dx  +\int_{\Omega}|H|^2\frac{d}{ds}\sqrt{\abs{g_s}}|_{s=0}dx\\
&=&2\int_{\Omega}H\cdot \Delta_g\varphi dV_g-2\int_{\Omega}g^{ip}(\dot{g}_{s,pq}|_{s=0}) g^{jq} (\partial^2_{ij} f\cdot H)  dV_g  +\int_\Omega|H|^2\nabla_gf\cdot \nabla_g\varphi dV_g\\
&=&2\int_{\Omega}H\cdot \Delta_g\varphi dV_g-4\int_{\Omega}H\cdot A_{ij}g^{ip}g^{jq}(\partial_p f\cdot \partial_q \varphi) dV_g
+\int_\Omega|H|^2\nabla_gf\cdot \nabla_g\varphi dV_g\\
&=&2\int_{\Omega}H\cdot \Delta \varphi dx -4\int_{\Omega}H\cdot A_{ij}g^{ip}\delta^{jq}(\partial_p f \cdot \partial_q \varphi) dx +\int_\Omega|H|^2\nabla f \cdot \nabla \varphi dx.
\end{eqnarray*}
Here in the last line, we have used the conformal invariance and $\Delta$ and $\nabla$ are the Laplace and the gradient with respect to the Euclidean metric on $D$.

If $\Omega=D$ and $\varphi$ has compact support in $D$, then integration by parts shows that $f$ being a Willmore immersion satisfies the Euler-Lagrange equation
\begin{equation}
	\label{eqn:el}
2\Delta H + 4{\rm div}(H\cdot A_{pq}g^{ip}\partial_i f) -{\rm div}(|H|^2\nabla f)=0.
\end{equation}

Now, for any $S\in \mathfrak{so}(\Real^n)$ and $c\in \Real^n$, let $e^{sS}$ be the family of rotations generated by $S$ and consider a special variation given by
\[
	f_s= e^{s S} f+sc.
\]
Hence,
\[
	\varphi=\dot{f_s}|_{s=0}= Sf+{c}
\]
and
\[
	W(f_s,\Omega)=W(f,\Omega), \qquad \text{for} \quad s\in (-\epsilon,\epsilon),
\]
which implies that $\frac{d}{ds}W(f_s,\Omega)|_{s=0}=0$. Since $f$ satisfies \eqref{eqn:el}, the boundary term that is involved in the integration by parts when deriving \eqref{eqn:el} must vanish, in spite of the fact that $\varphi$ is not supported in $\Omega$. 

If we take $\Omega=D_r$ for $0<r<1$, we obtain
$$
2\int_{\partial D_r}(H\cdot \partial_r \varphi- \partial_r H\cdot \varphi) d\sigma - 4\int_{\partial D_r}(H\cdot A_{rq}) g^{iq}(\partial_i f\cdot \varphi)d\sigma + \int_{\partial D_r}|H|^2 \partial_r f\cdot \varphi d\sigma=0
$$
where $A_{rq}=A(\partial_r f, \partial_q f)$.

The proof is done by substituting $\varphi= Sf+c$ and translating $(r,\theta)$ coordinate to $(t,\theta)$ coordinate.
\end{proof}

\begin{rem}
	\label{rem:t}
When $f$ is a conformal and Willmore immersion from $[0,T]\times S^1$, we regard $f$ as a map from $D\setminus D_{e^{-T}}$. Taking $\Omega=D_{r_1}\setminus D_{r_2}$, we learn from the same computations above that the definition of $\tau_1$ and $\tau_2$ are independent of $t$.
\end{rem}

\section{3-circle lemma for Willmore energy}
\label{sec:3will}

For some $L>0$, we define for any $i\in \mathbb Z$
$$
Q_i=[(i-1)L,iL]\times S^1.
$$
Usually, for each $Q_i$, we attach some quantity $\Phi_i\geq 0$ (which is some integral over $Q_i$). The 3-circle lemma aims at showing that for each $q\in (0,2)$, there is $L_0(q)$ and if $L>L_0$, we have
\begin{equation}
	\label{eqn:qphi}
	\Phi_i\leq e^{-qL} (\Phi_{i-1}+\Phi_{i+1})
\end{equation}
holds for $i$ in a certain range.

The advantage of having \eqref{eqn:qphi} is that we can show for each $q'\in (0,q)$, the existence of $L_1(q,q')$ such that if $L>L_1$, we have
\[
	\text{either} \quad \Phi_{i-1}\geq e^{q'L} \Phi_i \qquad \text{or} \quad \Phi_{i+1}\geq e^{q' L} \Phi_i \\
\]
and
\begin{eqnarray*}
	\Phi_i \geq e^{q'L} \Phi_{i-1} & \Longrightarrow & \Phi_{i+1}\geq e^{q'L} \Phi_i \\
	\Phi_i \geq e^{q'L} \Phi_{i+1} & \Longrightarrow & \Phi_{i-1}\geq e^{q'L} \Phi_i.
\end{eqnarray*}
Hence \eqref{eqn:qphi} is very useful in proving exponential decay of certain quantity along a cylinder and $q'$ represents the speed of decay.

In this section, we will study the Willmore energy of a conformal Willmore immersion $f$ defined on 
\[
	Q=Q_1\cup Q_2\cup Q_3.
\]
The main result of this section is a 3-circle type lemma for the Willmore energy. More precisely,
\begin{thm}\label{3-circle.H}
	For any $m\in \mathbb Z_+$ and $q\in (0,2)$, there exist $L_0(m,q)>0$ and $\epsilon_0(m,q)>0$. If $L>L_0$, the following holds.
Let $f:Q\rightarrow \R^n$ be a conformal and Willmore immersion with $g=f^*(g_{\R^n})=e^{-2mt+2v}(dt^2+d\theta^2)$ for some smooth function $v$. If 
\begin{equation}
	\label{eqn:tau12}
\tau_1(f,c)=\tau_2(f,S)=0,\qquad \text{for any} \quad 	S\in \so, \s c\in \Real^n
\end{equation}
and
\begin{equation}
	\label{eqn:closeflat}
\int_{Q_i}|A|^2+\|\nabla v\|_{C^0(Q_i)}<\epsilon_0,\s {i=1,2,3},
\end{equation}
then
\begin{equation}
	\label{eqn:3CH}
\int_{Q_2} \abs{H}^2 dV_g \leq e^{-qL}\left( \int_{Q_1} \abs{H}^2 dV_g + \int_{Q_3} \abs{H}^2 dV_g \right).	
\end{equation}
\end{thm}

\subsection{Harmonic functions}

The assumption \eqref{eqn:closeflat} implies that the mean curvature is approximately a harmonic function. We need the following 3-circle property of a harmonic function.

%
%
%
%
\begin{lem}\label{3-circle.harmonic}
	For each $m\in \mathbb Z_+$ and $q\in (0,2)$, there is $L_0(m,q)>0$ such that if $L>L_0(m,q)$ then the following holds.
Assume that $u\neq 0$ is a harmonic function defined on $Q$. The following expansion is well known
\begin{equation}
	\label{eqn:exp.harmonic}
u=a+bt+\sum_{k=1}^\infty\left((a_ke^{-kt}+b_ke^{kt})\cos k\theta+  (a_k'e^{-kt}+b_{k}'e^{kt})\sin k\theta\right).
\end{equation}
If $b_m=b_m'=0$,
then 
\begin{equation}
	\label{eqn:3C.harmonic}
\int_{Q_2}|u|^2e^{-2mt}dtd\theta < e^{-qL} \left( \int_{Q_1} \abs{u}^2 e^{-2mt} dtd\theta + \int_{Q_3} \abs{u}^2 e^{-2mt} dtd\theta \right).
\end{equation}
\end{lem}
\proof
By a direct calculation, we have (for $i=1,2,3$)
\begin{eqnarray*}
&&\|ue^{-mt}\|^{2}_{L^{2}(Q_i)}\\
&=&2\pi\dint_{(i-1)L}^{iL}|a+bt|^{2}e^{-2mt}dt\\
&&+\pi\int_{(i-1)L}^{iL}\sum_{k=1}^{+\infty}|a_k e^{-kt}+b_k e^{k t}|^{2}e^{-2mt}dt
+\pi\int_{(i-1)L}^{iL}\sum_{k=1}^{+\infty}|a_k' e^{-kt}+b_k' e^{kt}|^{2}e^{-2mt}dt\\
&=&2\pi\int_{(i-1)L}^{iL}|a+bt|^{2}e^{-2mt}dt\\
&&+\frac{\pi}{2}\sum_{k=1}^\infty{\Big(
}|a_k|^2\frac{1-e^{-2(k+m) L}}{(k+m)}e^{-2(k+m)(i-1)L}
+|b_k|^2\frac{e^{2(k-m) L}-1}{k-m}e^{2(k-m)(i-1)L}+\\
&&\s\s\s\s\s2\frac{a_k b_k}{m}(1-e^{-2mL})e^{-2m(i-1)L}{\Big)}\\
&&+\frac{\pi}{2}\sum\limits_{k=1}^\infty{\Big(
}{|a_k'|}^2\frac{1-e^{-2(k+m) L}}{(k+m)}e^{-2(k+m)(i-1)L}
+{|b_k'|}^2\frac{e^{2(k-m) L}-1}{k-m}e^{(k-m)(i-1)L}\\
&&\s\s\s\s\s2\frac{a_k' b_k'}{m}(1-e^{-2mL})e^{-2m(i-1)L}{\Big)}.\\
\end{eqnarray*}
In the above computation, when $k=m$, the expression $\frac{e^{2(k-m)L}-1}{k-m}$ does not make sense. However, since we have assumed $b_m=b'_m=0$, this is not a problem.

By setting $C_k, D_k, E_k$ and $C'_k, D'_k, E'_k$ properly, we can rewrite the last line in the above equation
\begin{eqnarray*}
	\norm{u e^{-mt}}_{L^2(Q_i)}^2
&:=& A_i + \sum_{k=1} \left( C_k e^{-2(k+m)(i-1)L} + D_k e^{2(k-m)(i-1)L} + E_k e^{-2m(i-1)L} \right)\\
&& + \sum_{k=1} \left( C'_k e^{-2(k+m)(i-1)L} + D'_k e^{2(k-m)(i-1)L} + E'_k e^{-2m(i-1)L} \right)\\
&:=& A_i + \sum_{k=1}^\infty B_{i;k} + \sum_{k=1}^\infty  B'_{i;k} \\
&:=&A_i+B_i+B_i'.
\end{eqnarray*}

{\bf Step 1. } Claim: there exists $C(m)>0$ such that
\begin{equation}
	\label{eqn:Bk}
B_{2;k}\leq C(m)e^{-2L}(B_{1;k}+B_{3;k}).
\end{equation}
Summing over $k$, we obtain
\begin{equation}
	\label{eqn:B}
B_{2}\leq C(m)e^{-2L}(B_{1}+B_{3}).
\end{equation}
The proof of the claim is divided into several cases.

{\noindent\it Case 1:} $k<m$.
By $k\geq 1$ and the Schwarz inequality,
	\begin{eqnarray*}
		\abs{\frac{2a_kb_k}{m}} &\leq& \frac{\sqrt{m^2-1}}{m} \abs{2\frac{a_k}{\sqrt{m+k}}\frac{b_k}{\sqrt{m-k}}}\\
					&\leq& \frac{\sqrt{m^2-1}}{m} \left( \frac{\abs{a_k}^2}{m+k} + \frac{\abs{b_k}^2}{m-k} \right).
	\end{eqnarray*}
Then
$$
(1-e^{-2mL})\abs{\frac{2a_kb_k}{m}}\leq \frac{\sqrt{m^2-1}}{m} \left( (1-e^{-2(k+m)L})\frac{\abs{a_k}^2}{m+k} + \frac{1-e^{-2mL}}{1-e^{2(k-m)L}}(1-e^{2(k-m)L})\frac{\abs{b_k}^2}{m-k} \right).
$$
	By the definition of $C_k,D_k$ and $E_k$, for $L$ sufficiently large, we have $c(m)\in (0,1)$ such that
	\[
		\abs{E_k}\leq (1-c(m)) (C_k+D_k),
	\]
which implies
$$
C_k+D_k+E_k= c(m)(C_k+D_k)+(1-c(m))(C_k+D_k)+E_k\geq c(m)(C_k+D_k).
$$	
Hence, 
	\begin{eqnarray*}
		B_{2;k}&=& C_ke^{-2(k+m)L}+D_ke^{2(k-m)L}+E_ke^{-2mL} \\
		       &\leq& e^{-2L}(C_k+D_k) + (1-c(m))e^{-2L}(C_k+D_k)\\
		       &\leq& 2 e^{-2L}(C_k+D_k) \\
		       &\leq& C(m) e^{-2L}(C_k+D_k+E_k)\\
		       &=& C(m) e^{-2L} B_{1;k}.
	\end{eqnarray*}
	Here $C(m)$ is a constant depending only on $m$ that may vary from line to line.

{\noindent\it {Case 2:}} $k=m$.

Note that $D_k=E_k=0$. We have
\begin{equation*}
B_{2;k}=e^{-2(k+m)L} C_k\leq e^{-2L}C_k=e^{-2L}B_{1;k}.
\end{equation*}

{\noindent\it Case 3:} $k>m$ and $E_k\geq 0$.

When $E_k\geq 0$, we know $D_k e^{2(k-m)(2L)}\leq B_{3;k}$, which implies that
\begin{eqnarray*}
B_{2;k}&=&C_ke^{-2(k+m)L}+D_ke^{2(k-m)L}+E_ke^{-2mL}\\
&=&C_ke^{-2(k+m)L}+e^{-2(k-m)L}D_ke^{4(k-m)L}+E_ke^{-2mL}\\
&\leq&e^{-2L}(C_k+E_k)+e^{-2L}(D_ke^{2(k-m)L\cdot2})\\
&\leq& e^{-2L}(B_1^k+B_3^k).
\end{eqnarray*}

{\noindent\it Case 4:} $k>m$  and $E_k<0$.

Choose $L_0$ such that for $L>L_0$,
$$
2e^{-2(k+m)L}\leq 2e^{-2L -2 mL}<e^{-2mL} \quad \text{and} \quad 2e^{-2(k-m)L-4mL}< e^{-2mL}.
$$
Together with $E_k<0$, we obtain
$$
(e^{-2(k+m)L}+e^{-2(k-m)L-4mL})E_k>e^{-2mL}E_k,
$$
which allows us to estimate
\begin{eqnarray*}
B_{2;k}&=&C_ke^{-2(k+m)L}+D_ke^{2(k-m)L}+E_ke^{-2mL}\\
&\leq&
e^{-2(k+m)L}(C_k+E_k)
+e^{-2(k-m)L}(D_ke^{2(k-m)\cdot2L}
+E_ke^{-2m\cdot2L})\\
&\leq&
e^{-2(k+m)L}(C_k+D_k+E_k)+e^{-2(k-m)L}(C_ke^{-(k+m)\cdot 2L}+D_ke^{2(k-m)\cdot2L}
+E_ke^{-2m\cdot2L})\\
&=& e^{-2(k+m)L}B_{1;k}+e^{-2(k-m)L}B_{3;k}\\
&\leq&e^{-2L}(B_{1;k}+B_{3;k}).
\end{eqnarray*}
Here in the last line above we used the fact that $B_{i;k}$ is nonnegative, which is obvious if we go back to its definition.

{\noindent\bf Step 2.} We show that for any $q<2$, we can choose $L_0$,  such that for any $L>L_0$, there holds
\begin{equation}
	\label{eqn:step2}
A_2\leq e^{-qL}(A_1+A_3).
\end{equation}

Setting
\begin{eqnarray*}
c_1&=&|a|^2+\frac{ab}{m}+\frac{|b|^2}{2m^2}=|a+\frac{b}{2m}|^2+\frac{|b|^2}{4m^2} \\
c_2&=&\frac{|b|^2}{m}+2ab \\
c_3&=&|b|^2,
\end{eqnarray*}
and
\[
	F(t)= \frac{e^{-2mt}}{2m} (c_1+c_2 t +c_3 t^2),
\]
we verify that $-F(t)$ is a primitive function of $(a+bt)^2 e^{-2mt}$ such that
\begin{eqnarray*}
	A_1&=&F(0)-F(L)\\
	A_2&=&F(L)-F(2L)\\
	A_3&=&F(2L)-F(3L).
\end{eqnarray*}
It is easy to check the following easy facts
\begin{enumerate}
	\item $c_1+c_2 t+c_3 t^2 \geq 0$ for any $t\in \Real$, hence $F(t)\geq 0$ for all $t\in \Real$;
	\item $c_1\geq 0$;
	\item $c_2L \leq 2m c_1 L$;
	\item $c_3L^2 \leq 4m^2 c_1 L^2$. 
\end{enumerate}
Using these properties, we have
\begin{eqnarray*}
	F(L)&=& \frac{e^{-2mL}}{2m}(c_1+c_2L+c_3L^2) \leq C(m) e^{-2mL} L^2 c_1 \leq C(m) L^2 e^{-2L} F(0)\\
	F(3L)&=&\frac{e^{-6mL}}{2m}(c_1+3c_2L+9c_3L^2) \leq C(m) e^{-2mL} L^2 c_1 \leq C(m) L^2 e^{-2L} F(0).
\end{eqnarray*}
Hence for a fixed $q<2$, we may choose $L_0$ such that for any $L>L_0$, we have
$$
(1+e^{-qL})F(L)\leq\frac{e^{-qL}}{2}F(0)
,\s e^{-qL}F(3L)\leq \frac{e^{-qL}}{2}F(0).
$$
Taking the sum of the above inequalities and noticing that $F(2L)\geq 0$, we get
$$
(1+e^{-qL})F(L)+e^{-qL}F(3L)\leq e^{-qL}F(0)\leq (1+e^{-qL})F(2L)+e^{-qL}F(0),
$$
which implies that
$$
F(L)-F(2L)\leq e^{-qL}(F(0)-F(L)+F(2L)-F(3L)).
$$
The proof for \eqref{eqn:step2} is done.

Finally, by the same argument as in the Step 1, we also have
\[
	B_2'\leq C(m) e^{-2L} (B_1'+B_3').
\]
Together with \eqref{eqn:B}, we can choose $L_0$ such that if $L>L_0$, we have
\[
	B_2+B_2'\leq e^{-qL}(B_1+B_1'+B_3+B_3').
\]
and the inequality can be made strict as long as the right hand side is not zero.
For fixed $q\in (0,2)$, we use Step 2 for $q'$ (satisfying $q<q'<2$) to see 
\[
	A_2\leq e^{-qL}(A_1+A_3)
\]
and again the inequality can be made strict as long as $A_1+A_3$ is not zero.
Since we have assumed that $u$ is not identically zero, at least one of the above inequalities are strict. The proof of \eqref{eqn:3C.harmonic} is done by adding them up.
\endproof

\subsection{Proof of Theorem \ref{3-circle.H}}

\proof 
Given $m$ and $q$, let $L_0$ be determined by Lemma \ref{3-circle.harmonic}. Argue by contradiction. Assume that there exist $L>L_0$ and a sequence of $f_k$  with $g_k=f_k^*(g_{\Real^n})=e^{-2mt+2v_k}(dt^2+d\theta^2)$ satisfying
\begin{equation}
	\label{eqn:smallAv}
	\int_Q \abs{A_k}^2 dV_{g_k}  + \norm{\nabla v_k}_{C^0(Q)} \to 0
\end{equation}
but
\begin{equation*}
\int_{Q_2} \abs{H_k}^2 dV_{g_k} > e^{-qL}\left( \int_{Q_1} \abs{H_k}^2 dV_{g_k} + \int_{Q_3} \abs{H_k}^2 dV_{g_k} \right),
\end{equation*}
that is,
\begin{equation}\label{3.W}
\int_{Q_2}|H_k|^2e^{-2mt+2v_k}dtd\theta > e^{-q L} \left( \int_{Q_1}|H_k|^2e^{-2mt+2v_k}dtd\theta + \int_{Q_3}|H_k|^2e^{-2mt+2v_k}dtd\theta \right)
\end{equation}
for $k$ sufficiently large.


After a scaling for each $f_k$, by \eqref{eqn:smallAv}, we assume that $v_k$ converges to zero uniformly on $Q$. We still denote the scaled maps by $f_k$. This will not cause any problem because every assumption and the conclusion of the theorem is invariant after the scaling.

The boundedness of $v_k$ on $Q$ allows us to use the $\epsilon$-regularity estimate (\cite[Theorem I.1]{B-R}) to control any derivative of $f_k$ on any compact set of $Q$. Hence, up to translation and taking subsequence, we may assume that $f_k$ converges locally smoothly to a conformal Willmore immersion $f_\infty $. By \eqref{eqn:smallAv}, the second fundamental form of $f_\infty$ vanishes and the induced metric $g_\infty = e^{-2mt}(dt^2+d\theta^2)$. The vanishing of the second fundamental form implies that the image of $f_\infty $ lies in a plane and after a rotation of $\Real^n$ (which may flip the orientation if necessary), we may assume that $f_\infty (t,\theta)=(u(t,\theta),v(t,\theta),0,...,0)$. Since $f_\infty $ is conformal, the explict formula of $g_\infty $ implies that
\[
	\abs{\partial_t u}^2+ \abs{\partial_t v}^2 = \abs{\partial_\theta u}^2+ \abs{\partial_\theta v}^2= e^{-2mt} \quad \text{and} \quad \partial_t u \cdot \partial_\theta u +\partial_t v\cdot \partial_\theta v=0.
\]
It follows from elementary complex analysis that after a rotation of the domain (i.e. $\theta\mapsto \theta+ \theta_0$) if necessary and a translation in $\Real^n$, $f_\infty $ is uniquely determined
\[
	f_\infty (t,\theta)= \frac{e^{-mt}}{m} (\cos m\theta, \sin m\theta, 0, \cdots ,0).
\]

We observe that all assumptions and the conclusion of Theorem \ref{3-circle.H} is invariant under scaling, translation and rotation. Hence, we may assume that $f_k$ converges to $f_\infty $ locally smoothly in $Q$.

\begin{rem}
	\label{rem:tau1}
	It is important to note that $\tau_1(f,c)=0$ (for all $c\in \Real^n$) is invariant under scaling, translation and rotation, while $\tau_2(f,S)=0$ (for all $S\in \so$) is invariant under scaling and rotation, but not translation. However, if $\tau_1(f,c)=0$ is assumed, then $\tau_2(f,S)=0$ becomes invariant under translation. These facts can be verified directly using the definitions.
	It will be clear that for the rest of the proof only $\tau_2(f,S)=0$ is used. However, we still need to assume $\tau_1(f,c)=0$ in the theorem to ensure that $\tau_2(f,S)=0$ is invariant under translation.
\end{rem}

Set
$$
\phi_k=H_k/\| H_ke^{-mt}\|_{L^2(Q_2)}.
$$
Here the $L^2$ norm is taken with respect to the flat metric $dt^2+d\theta^2$.  By \eqref{3.W}, we have
\begin{equation}\label{3.W2}
\int_{Q_2}|\phi_k|^2e^{-2mt+2v_k}dtd\theta> e^{-q L} \left(\int_{Q_1}|\phi_k|^2e^{-2mt+2v_k}dtd\theta + \int_{Q_3}|\phi_k|^2e^{-2mt+2v_k}dtd\theta \right)
\end{equation}
which implies
\begin{equation}
	\label{eqn:l2}
\|\phi_ke^{-mt}\|_{L^2(Q)}<C.
\end{equation}
If we rewrite the equation of $H_k$ (see \eqref{eqn:el}) in the form of a linear equation (whose coefficients are allowed to depend on $H_k$, $A_k$ and their derivatives), we obtain the following equation for $\phi_k$
\begin{equation}\label{equ.H}
\Delta\phi_k+\alpha_k\cdot\nabla\phi_k+\beta_k\phi_k=0.
\end{equation}
Thanks to \eqref{eqn:smallAv}, we have
\[
\|\alpha_k\|_{C^l(K)}+\|\beta_k\|_{C^l(K)}\rightarrow 0
\]
for any $l\in \mathbb N$ and any compact subset $K \subset Q$. The equation \eqref{equ.H} and the $L^2$ bound \eqref{eqn:l2} imply any $C^l$ estimates of $\phi_k$ on $K$. By taking subsequence, we assume $\phi_k$ converges in $C^\infty_{loc}(Q)$ to some harmonic function denoted by $\phi$. By our definition of $\phi_k$,
$$
\|\phi e^{-mt}\|_{L^2(Q_2)}=1.
$$
By the Fatou's lemma and \eqref{eqn:smallAv}, \eqref{3.W2}, we have
\begin{equation}\label{3.W3}
\int_{Q_2}|\phi|^2e^{-2mt}dtd\theta\geq e^{-q L}
\left(\int_{Q_1}|\phi|^2e^{-2mt}dtd\theta+
\int_{Q_3}|\phi|^2e^{-2mt}dtd\theta\right).
\end{equation}
In order to get a contradiction to the three circle property given by Lemma \ref{3-circle.harmonic}, it suffices to show that $b_m=b'_m=0$ in the following expansion of harmonic function $\phi$
\[
\phi=a+bt+\sum_{k=1}^\infty\left((a_ke^{-kt}+b_ke^{kt})\cos k\theta+  (a_k'e^{-kt}+b_{k}'e^{kt})\sin k\theta\right).
\]

Since $H_k$ is a normal vector,
$$
\phi\cdot \partial_t (f_\infty )=\lim_{k\rightarrow+\infty}\frac{{H}_k\cdot (\partial_t f_k)}{\| H_ke^{-mt}\|_{L^2(Q_2)}}=0.
$$
Similarly, we have $\phi\cdot \partial_\theta(f_\infty )=0$. Hence, we may assume
$$
\phi=(0,0,\phi_3,\cdots,\phi_n).
$$
Taking the limit $k\to \infty $, we get from $\tau_2(f_k,S)=0$ that
\begin{equation}\label{eqn:limit}
0=\int_{\{t\}\times S^1}\left(\phi\cdot \partial_t (Sf_\infty )-\partial_t \phi\cdot (Sf_\infty )\right) d\theta.
\end{equation}
Letting $S=(s_{ij})$ with 
$$
s_{ij}=\left\{\begin{array}{ll}1&i=1,j=n\\
              -1&i=n,j=1\\
              0&others,
             \end{array}\right.
$$ 
we obtain from \eqref{eqn:limit} that
\[
	0 = \int_{\set{t} \times S^1} \phi_n (m e^{-mt}\cos m\theta) - \partial_t \phi_n (-e^{-mt}\cos m\theta) d\theta.	
\]
This implies that $b_m=0$ for the expansion of $\phi_n$. 
By choosing different $S$, the same argument shows that $b_m=b'_m=0$ for the expansion of $\phi_3,\cdots ,\phi_n$. This concludes the proof of Theorem \ref{3-circle.H}.
\endproof

Consequently, following the discussion at the beginning of Section \ref{sec:3will}, we have the following
\begin{cor}\label{pro:decayH}
For any $q\in (0,2)$ and $m\in \mathbb Z_+$,  there exist $L_0(m,q),\delta_0(m,q)>0$, such that if $L>L_0$ then the following holds.
	Let $f:[0,kL]\times S^1\rightarrow\R^n$ be a conformal and Willmore immersion with $f^*(g_{\R^n})=e^{-2mt+2v}(dt^2+d\theta^2)$. Assume $\tau_1(f,S)=0$ and $\tau_2(f,c)=0$ for any $S\in \so$ and $c\in\R$. If
$$
\|\nabla v\|_{L^\infty(Q_i)}+\int_{Q_i}|A|^2 dV_g<\delta_0,\qquad \forall\s 1\leq i\leq k, 
$$
then we have
\begin{equation}\label{decay.H.ine}
\int_{Q_i}|H|^2dV_{g}<C(e^{-iqL}\int_{Q_1}|H|^2dV_{g}+e^{-(k-i)qL}\int_{Q_k}|H|^2dV_{g}).
\end{equation}
\end{cor}

\section{3-circle lemma for the second fundamental form}
\label{sec:A}

In this section, we prove Theorem \ref{thm:out}. It follows directly from the following theorem.
\begin{thm}\label{thm:main}
For any $q\in (0,2)$ and $m\in \mathbb Z_+$,  there exist $L_2(m,q),\delta_2(m,q)>0$, such that if $L>L_2$ then the following holds.
	Let $f:[0,kL]\times S^1\rightarrow\R^n$ be a conformal and Willmore immersion with $f^*(g_{\R^n})=e^{-2mt+2v}(dt^2+d\theta^2)$. Assume $\tau_1(f,S)=0$ and $\tau_2(f,c)=0$ for any $S\in \so$ and $c\in\R$. If
$$
\|\nabla v\|_{L^\infty(Q_i)}+\int_{Q_i}|A|^2 dV_g<\delta_2,\qquad \forall\s 1\leq i\leq k, 
$$
then we have the following decay estimate for the second fundamental form $A$
\begin{equation}
	\label{eqn:decayA}
\int_{Q_i}|A|^2dV_{g}<C\left(e^{-iqL}\int_{Q_1\cup Q_2\cup Q_3}|A|^2dV_{g}+e^{-(k-i)qL}\int_{Q_{k-2}\cup Q_{k-1}\cup Q_k}|A|^2dV_{g}\right),
\end{equation}
where $C$ depends on $m$, $q$, $L$.
\end{thm}

Assuming Theorem \ref{thm:main}, for $q,m$ as above, take $\delta_2$ and $L_2$ as given in Theorem \ref{thm:main} and fix some $L>L_2\geq 1$ such that $T=kL$ for some $k\in \mathbb Z_+$. Then (a)-(c) in Theorem \ref{thm:out} implies that all assumptions of Theorem \ref{thm:main} hold with $\delta_1=\delta_2/(2L)$. Since
\[
	\int_{Q_1\cup Q_2\cup Q_3} \abs{A}^2 dV_g, \int_{Q_{k-2}\cup Q_{k-1}\cup Q_k} \abs{A}^2 dV_g\leq C(L) \tilde{\delta},
\]
\eqref{eqn:decayaa} is a consequence of \eqref{eqn:decayA}. This proves Theorem \ref{thm:out}.

The proof of Theorem \ref{thm:main} consists of two parts. The first part shows that when $H$ is much smaller than $A$, the 3-circle holds for the $L^2$ norm of $A$ (Theorem \ref{3.circle.A}). The second part combines this with Corollary \ref{pro:decayH} to prove Theorem \ref{thm:main}.

\subsection{When $H$ is small.}
The Gauss map ${\bf n}$ takes value in the Grassmannian $G(2,n)$.
There is a natural embedding of $G(2,n)$ into $\Lambda^2(\Real^n)$, the $2$-nd exterior power of $\Real^n$. The induced metric from this embedding agrees with the metric as the symmetric space $O(n)/O(2) \times O(n-2)$. We denote the second fundamental form of $G(2,n)$ by $A_{G(2,n)}$. The main result in \cite{R-V} shows
\begin{equation}
	\label{eqn:RV}
	\tau({\bf n}):=\Delta {\bf n} + A_{G(2,n)}(d{\bf n},d{\bf n})= \nabla^{\perp} H,
\end{equation}
where $H$ is the mean curvature and $\nabla^\perp H$ is the covariant derivative of the normal bundle. The equation \eqref{eqn:RV} is to be understood with an isomorphism (see Section 2 of \cite{R-V}) that identifies $\nabla^{\perp} H$ with a section of ${\mathbf n}^{-1}(TG(2,n))$.

We regard \eqref{eqn:RV} as a perturbed harmonic map equation. In the neck analysis of harmonic maps, it is important to have some Pohozaev identity that relates the tangential energy of $\mathbf n$ with the radial energy. In our setting, this is achieved by the following lemma.

\begin{lem}
	\label{lem:poho.A} Let $f:(t_1,t_2)\times S^1\to \Real^n$ be any conformal immersion with $f^*(g_{\Real^n})=e^{2u}(dt^2+d\theta^2)$. Let ${\mathbf n},H,A$ be the Gauss map, mean curvature and second fundamental form of $f$ respectively. Then for some universal constant $C$,
\begin{equation}\label{Po.A}
	\left|\left|\partial_t {\bf n}\right|^2-\left|\partial_\theta {\bf n}\right|^2\right|\leq Ce^{2u}|A||H|, \qquad \forall (t,\theta)\in  (t_1,t_2)\times S^1.
\end{equation}
\end{lem}
\begin{proof}
	The proof is by direct computation.

We set $e_1=e_t=e^{-u}\partial_t f$, $e_2=e_\theta=e^{-u}\partial_\theta f$, and select $e_3$, $\cdots$, $e_n$ (depending smoothly on $(t,\theta)$) such that $\{e_i\}$ is an orthonormal basis of $\R^n$. The Gauss map is given by
\[
	{\bf n}(t,\theta)= e_t \wedge e_\theta.
\]
By the definition of the second fundamental form $A$ and the Christoffel symbol $\Gamma$ of the induced connection, we have
\begin{eqnarray*}
e^{-u}\frac{\partial e_t}{\partial t}&=&\Gamma_{tt}^\theta e_\theta+ e^{-2u}A_{tt}^\alpha e_\alpha \\
e^{-u}\frac{\partial e_\theta}{\partial t}&=&\Gamma_{\theta t}^te_t+e^{-2u}A_{\theta t}^\alpha e_\alpha \\
e^{-u}\frac{\partial e_t}{\partial \theta}&=&\Gamma_{t\theta}^\theta e_\theta+e^{-2u}A_{t\theta}^\alpha e_\alpha\\
e^{-u}\frac{\partial e_\theta}{\partial \theta}&=&\Gamma_{\theta\theta}^te_t+e^{-2u}A_{\theta\theta}^\alpha e_\alpha,
\end{eqnarray*}
where $\alpha=3,\cdots,n$ and $A_{tt}$ is short for $A(f_t,f_t)$, etc. By definition, we have
$$
H=e^{-2u}(A_{tt}^\alpha+A_{\theta\theta}^\alpha)e_\alpha,\s |A|^2= e^{-4u}\sum_\alpha(|A_{tt}^\alpha|^2+2|A_{t\theta}^\alpha|^2+|A_{\theta\theta}^\alpha|^2).
$$
By a direct calculation, we have
$$
\frac{\partial(e_t\wedge e_\theta)}{\partial t}= e^{-u}\left(A_{tt}^\alpha e_\alpha\wedge e_\theta+A_{\theta t}^\alpha e_t\wedge e_\alpha\right)
$$
and
$$
\frac{\partial(e_t\wedge e_\theta)}{\partial \theta}= e^{-u}\left(A_{t\theta}^\alpha e_\alpha\wedge e_\theta+A_{\theta\theta}^\alpha e_t\wedge e_\alpha\right).
$$
Hence,
$$
\abs{\partial_t {\bf n}}^2 - \abs{\partial_\theta {\bf n}}^2 =e^{-2u}\sum_\alpha|A_{tt}^\alpha|^2-|A_{\theta\theta}^\alpha|^2=e^{2u} \sum_\alpha e^{-2u}(A_{tt}^\alpha+A_{\theta\theta}^\alpha) e^{-2u}(A_{tt}^\alpha-A_{\theta\theta}^\alpha),
$$
from which \eqref{Po.A} follows.
\end{proof}

Next, we use the 3-circle argument to prove
\begin{thm}\label{3.circle.A}
	For any $q\in (0,2)$ and $\beta>0$, there exists $L_1(q,\beta)$ such that if $L>L_1$, the following holds. There exists $\epsilon(q,\beta)>0$ and $\delta(q,\beta)>0$ such that if $f$ is a conformal and Willmore immersion defined on $Q=Q_1\cup Q_2\cup Q_3$ satisfying
	\begin{enumerate}
		\item $\abs{\nabla u}\leq \beta$ where $f^*(g_{\Real^n})=e^{2u}(dt^2+d\theta^2)$;
		\item $\int_{Q} \abs{H}^2 dV_g\leq \delta \int_Q \abs{A}^2 dV_g$;
		\item $\norm{A}_{L^2(Q)}<\epsilon$;
	\end{enumerate}
then  
\[
	\int_{Q_2} \abs{A}^2 dV_g \leq e^{-qL} \left( \int_{Q_1} \abs{A}^2 dV_g+ \int_{Q_3}\abs{A}^2 dV_g \right).
\]
\end{thm}

\proof 
If otherwise, then we can find $f_k$ satisfying (1), such that 
\begin{equation}
	\label{eqn:ass}
	\lim_{k\to \infty }\frac{1}{\|A_k\|^2_{L^2(Q)}} \int_{Q} \abs{H_k}^2 dV_{g_k}= \lim_{k\to \infty }\norm{A_k}_{L^2(Q)}=0
\end{equation}
and
\begin{equation}
	\label{eqn:contra}
\int_{Q_2}|A_k|^2 dV_{g_k} \geq e^{-qL}\left(\int_{Q_1}|A_k|^2 dV_{g_k} +\int_{Q_3}|A_k|^2 dV_{g_k} \right).
\end{equation}
Similar to the proof of Theorem \ref{3-circle.H}, by a scaling of $f_k$ and (1), we first assume that $u_k$'s are bounded. Then by \eqref{eqn:ass}, we use the $\epsilon$-regularity of Willmore surface to show that $f_k$ (after translation) converges smoothly locally in $Q$. Moreover, we know $A_k$ converges to $0$ and $u_k$ converges to some smooth function $u$. 

We multiply both sides of \eqref{eqn:RV} by $\frac{1}{\norm{ \nabla {\bf n}_k}_{L^2(Q_2)}}$ and set $\phi_k= \frac{{\bf n}_k-c_k}{ \norm{ \nabla {\bf n}_k}_{L^2(Q_2)}}$ where $c_k$ is the average of ${\bf n}_k$ over $Q_2$, to get
\begin{equation}
	\label{eqn:scale}
	\Delta \phi_k + A_{G(2,n)}(d{\bf n}_k, d\phi_k)= \frac{\nabla^\perp H_k}{\norm{\nabla {\bf n}_k}_{L^2(Q_2)}}.
\end{equation}
Observe that $H_k$ satisfies \eqref{eqn:el} and we have control over any $C^l$ norm of $H_k$, $f_k$ and $A_k$ so that we can rewrite \eqref{eqn:el} as a linear equation of $\phi_k$ with bounded coefficients. The elliptic interior estimate then implies that
\[
	\sup_{ [\eta,3L-\eta]\times S^1} \abs{\nabla^\perp H_k} \leq C(\eta) \norm{H_k}_{L^2(Q)} .
\]
Thanks to \eqref{eqn:ass} and \eqref{eqn:contra}, by taking $k\to \infty $, we obtain a (vector-valued) harmonic function $\phi$ as the limit of $\phi_k$ defined on $Q$. By \eqref{eqn:contra} and the choice of scaling constant, we have
\[
	\int_{Q_2} \abs{\nabla\phi}^2 dtd\theta =1, \quad \int_{Q} \abs{\nabla \phi}^2 dtd\theta <C.
\]

It follows from \eqref{Po.A} and \eqref{eqn:ass} that
\[
	\int_{Q_2} \abs{\partial_t \phi}^2 - \abs{\partial_\theta \phi}^2 dtd\theta =0.
\]
For such harmonic functions, there is a version of three circle lemma which implies the existence of $L_1$ such that if $L>L_1$
\begin{equation}
	\label{eqn:inappendix}
	\int_{Q_2} \abs{\nabla\phi}^2 dt\theta < e^{-qL} \left( \int_{Q_1} \abs{\nabla\phi}^2 dt\theta + \int_{Q_3} \abs{\nabla\phi}^2 dtd\theta \right).
\end{equation}
We postpone the proof of this claim to the appendix (see Lemma \ref{lem:appendix}).
The inequality \eqref{eqn:inappendix} contradicts \eqref{eqn:contra} (by the smooth convergence on $Q_2$ and the Fatou's lemma) and finishes the proof.
\endproof

\subsection{Proof of Theorem \ref{thm:main}.}
Now, we combine the decay of the mean curvature (Corollary \ref{pro:decayH}) with Theorem \ref{3.circle.A} to show the decay of the second fundamental form. 

\proof 
Given $m$ and $q$, fix some $q'\in (q,2)$ and let $\beta=m+1$. Thereom \ref{3.circle.A} determines $L_1$, $\epsilon$ and $\delta$ depending on $q'$ and $\beta$. Corollary \ref{pro:decayH} determines $L_0$ and $\delta_0$ depending on $q'$ and $m$. Fix $\delta_2<\min (\delta_0,\epsilon/3)$ and $L_2>\max(L_1,L_0)$ satisfying
\begin{equation}
	\label{eqn:smart}
	e^{- (q'-q)L}<1/2 \qquad \text{for} \quad L>L_2.
\end{equation}
For $1<i<k$, if
\begin{equation}
	\label{eqn:good}
\int_{Q_{i-1}\cup Q_i\cup Q_{i+1}}|A|^2 dV_g \leq \frac{1}{\delta}\int_{Q_{i-1}\cup Q_i\cup Q_{i+1}} \abs{H}^2 dV_g.
\end{equation}
Then \eqref{eqn:decayA} for $i$ holds due to Corollary \ref{pro:decayH}. 

In what follows, we assume \eqref{eqn:good} is not true for $i$, hence it verifies the assumption (2) in Theorem \ref{3.circle.A}. The assumption (1) holds for $\beta=m+1$ and $\delta_2<1$, and the assumption (3) is obvious because $\delta_2<\epsilon/3$. 
Hence Theorem \ref{3.circle.A} (for $q'$ and $Q=Q_{i-1}\cup Q_i\cup Q_{i+1}$) then implies that 
\[
	\int_{Q_i} \abs{A}^2 dV_g \leq e^{-q'L} \left( \int_{Q_{i-1}} \abs{A}^2 dV_g + \int_{Q_{i+1}}\abs{A}^2 dV_g \right).
\]
By \eqref{eqn:smart}, we have either 
$\int_{Q_i}|A|^2 dV_g <e^{-qL}\int_{Q_{i-1}}|A|^2 dV_g$ or $\int_{Q_i}|A|^2 dV_g<e^{-qL}\int_{Q_{i+1}}|A|^2 dV_g$.
Without loss of generality, we assume the later one is true. Hence, the proof of \eqref{eqn:decayA} for $i$ is reduced to the proof of the same inequality with $i$ replaced by $i+1$. 

Now, we may repeat the above argument to consider $i+2$, $i+3$,$\cdots$, until either \eqref{eqn:good} holds for $i+l$, or $i+l=k$. In both cases, we have \eqref{eqn:decayA} true for $i+l$ and hence \eqref{eqn:decayA} holds for $i$.
\endproof

\section{Positive answers to Q1 and Q2}
\label{sec:decay}

In this section, we assume $f_k$ is a sequence of conformal and Willmore immersions satisfying A.1)-A.4). Our goal is to give positive answers to Q1 and Q2.
%

Before that, we prove a claim promised in Remark \ref{rem:a3}.

\begin{lem}
	\label{lem:remark} If $f_k:[0,T_k]\times S^1\to \Real^n$ can be extended to conformal and Willmore immersions of closed surfaces $\Sigma'_k$ with bounded Willmore energy and genus, then there exists some $\beta>0$ and $T>0$ depending on the sequence such that 
	\begin{equation}
		\label{eqn:beta}
\sup_k \|\nabla u_k\|_{L^\infty([T,T_k-T])}<\beta.
	\end{equation}
\end{lem}
\begin{proof}
	The Gauss-Bonnet theorem implies that $\norm{A_{\Sigma'_k}}_{L^2}$ is uniformly bounded. By Theorem 7.1 in \cite{K-L} (or Theorem 6.4 in \cite{L-W-Z}), for any sequence $t_k$ satisfying $t_k\to +\infty $ and $T_k-t_k\rightarrow+\infty$, we can find $\lambda_k$ and $c_k$, such that $\hat f_k=\lambda_k(f_k(t_k+t,\theta)-c_k)$ converges in $W^{2,2}$ to a 
conformal map. According to Helein's convergence theorem, at this point, we have $|u_k|$ is bounded on any $(-T,T)\times S^1$, and further, due to the $\epsilon$-regularity (cf. \cite[Theorem 2.10]{K-S},\cite[Theorem I.5]{R}), 
$\hat f_k$ converges smoothly. Hence, if $\hat{u}_k$ is defined by $e^{2\hat{u}_k}(dt^2+d\theta^2)= \hat{f}_k^*(g_{\Real^n})$, we have
\[
	\sup_k \sup_{\set{0}\times S^1} \abs{\nabla\hat{u}_k}<\infty.
\]
If \eqref{eqn:beta} were not true, then after taking a subsequence, we can find $t_k$, such that $t_k\rightarrow+\infty$, $T_k-t_k\rightarrow+\infty$, and 
$\sup_{\{t_k\}\times S^1}|\nabla u_k|\rightarrow+\infty$. This is a contradiction because
\[
	\abs{\nabla \hat{u}_k}(t,\theta) = \abs{\nabla u_k}(t+t_k,\theta).
\]
\end{proof}

Now, let's discuss Q1 and Q2. We observe that the positive answers to Q1 an Q2 follow from the following claim:

{\bf Claim.} For any $\varepsilon>0$, there exist $T>0$ and $M>0$ such that for any $k>M$, we have
\[
	\int_{[t,t+1]\times S^1} \abs{A_k}^2 dV_k \leq C\varepsilon \left( e^{-q (t-T)}+e^{-q(T_k-T-t)} \right) \qquad \text{for} \quad t\in [T,T_k-T].
\]

The proof of the claim uses Theorem \ref{thm:out}. The assumption (a) follows from A.4) and Theorem \ref{thm:poho}. By A.2), we have $T'>0$ and $M'>0$ such that for $k>M'$,
\[
	\sup_{t\in [T',T_k-T']} \int_{[t,t+1]\times S^1} \abs{A_k}^2 dV_k< \min(\delta_1,\varepsilon),
\]
which is the assumption (c). The proof of the claim is then reduced to the verification of (b), which is the next lemma. Assuming the lemma, we can take $T=\max(T',T'')$ and $M=\max(M',M'')$ and apply Theorem \ref{thm:out} to the restriction of $f_k$ on $[T,T_k-T]\times S^1$ to see that the claim holds.

\begin{lem}
	\label{lem:b} Assume that $f_k$ satisfies A.1)-A.3). For any $\varepsilon>0$, there exist $T''>0$ and $M''>0$ such that for any $k>M''$, there is $m_k\in \mathbb Z\setminus \set{0}$ with $\abs{m_k}\leq \beta$ such that if $f_k^*(g_{\Real^n})=e^{-2m_kt + 2v_k}(dt^2+d\theta^2)$, then we have
\begin{equation}
	\label{eqn:remain}
	\sup_{[T'',T_k-T'']}\abs{\nabla v_k}<\varepsilon.
\end{equation}
\end{lem}

\begin{proof}
	We consider first an apparently weaker statement: for any $\varepsilon>0$, there exist $T''>0$ and $M''>0$ such that for any $k>M''$, for any $t_0\in [T'',T_k-T'']$, we have $m=m(k,t_0)\in \mathbb Z\setminus \set{0}$ and $f_k^*(g_{\Real^n})=e^{2u_k}(dt^2+d\theta^2)=e^{-2mt + 2v_k}(dt^2+d\theta^2)$ satisfying $\abs{m}\leq \beta$ and
	\[
		\sup_{\set{t_0}\times S^1} \abs{\nabla v_k}<\varepsilon. 
	\]
Assuming this ``weaker" statement, if we fix $k>M''$, by the continuity of $\nabla u_k$ and 
\[
	\sup_{\set{t_0}\times S^1} \abs{\partial_t(u_k)+ m(k,t_0))}<\varepsilon, \qquad \text{for all} \quad t_0\in [T'',T_k-T''],
\]
we know $m(k,t_0)$ is independent of $t_0$. Hence, the lemma follows from this ``weaker" statement, which we prove below by contradiction.  

Assume that there is $\varepsilon_0>0$, a subsequence $k_l$, such that if $u_l$ is defined by 
	\[
		f^*_{k_l}(g_{\Real^n})= e^{2u_l}(dt^2+d\theta^2),
	\]
	we have $t_l\in [l,T_k-l]$ satisfying
	\begin{equation}
		\label{eqn:tl}
		\sup_{\set{t_l}\times S^1}\abs{\partial_t (u_l+mt)}\geq \varepsilon_0 \qquad \text{for any} \quad m\in \mathbb Z\setminus \set{0},\s \abs{m}\leq \beta.
	\end{equation}
	For $\lambda_l=e^{-u_l(t_l,0)}$, we set $\hat{f}_l(t,\theta)=\lambda_l (f_{k_l}(t+t_l,\theta)- (2\pi)^{-1}\int_0^{2\pi} f_{k_l}(t_l,\theta)d\theta)$. Let $e^{2\hat{u}_l}$ be the conformal factor of $\hat{f}^*_l(g_{\Real^n})$. By our choice of $\lambda_l$ and A.3), $\hat{u}_l$ is uniformly bounded on $[-l_0,l_0]\times S^1$ when $l\geq l_0$. By A.2), if we take $l_0$ sufficiently large, then for $l>l_0$, we have
	\[
		\sup_{t\in [-l_0/2,l_0/2]} \int_{[t,t+1]\times S^1} \abs{A_l}^2 dtd\theta< \epsilon
	\]
	for any $\epsilon$. Therefore, the $\epsilon$-regularity provides any $C^k$ bound for $\hat{f}_l$ on $[-l_0/3,l_0/3]\times S^1$. Hence, by taking subsequence (still denoted by $\hat{f}_l$), we may assume that $\hat{f}_l$ converges locally smoothly to a limit $f_\infty $ defined on $\Real \times S^1$.

$f_\infty$ is a conformal Willmore immersion from $\Real \times S^1$ to $\R^n$ with vanishing second fundamental form. If $(f_\infty )^*(g_{\Real^n})=e^{2u}(dx^2+dy^2)$, then $u$ is a harmonic function defined on $\Real \times S^1$ satisfying $\abs{\nabla u}\leq \beta$. By the well known expansion for harmonic function
\begin{equation}
	\label{eqn:expharmonic}
	u(t,\theta)=-mt+b + \sum_{k=1}^\infty \left( (a_ke^{-kt}+b_k e^{kt})\cos k\theta + (a'_k e^{-kt}+b'_k e^{kt})\sin k\theta \right),
\end{equation}
we must have $u(t,\theta)=-m t+b$ for $m\in \Real$ with $\abs{m}\leq \beta$. By modifying $\lambda_l$, we may assume that $b=0$ and hence $u(t,\theta)\equiv -mt$.

Since $A_{f_\infty }\equiv 0$, the image of $f_\infty $ lies in a plane. By a rotation of $\Real^n$, we may assume that $f_\infty (t,\theta)=(w(t,\theta),v(t,\theta),0,\cdots ,0)$ and $w+iv$ is a holomorphic function. By $u(t,\theta)\equiv -mt$, we obtain
\[
	\abs{\partial_t w}^2+ \abs{\partial_\theta w}^2 = \abs{\partial_t v}^2+ \abs{\partial_\theta v}^2= e^{-2mt}, \qquad \text{for all} \quad (t,\theta)\in \Real \times S^1.
\]
This equation, together with a similar expansion of $w$ as \eqref{eqn:expharmonic}, we conclude that $m\in \mathbb Z\setminus \set{0}$. This is a contradiction to \eqref{eqn:tl} because $\nabla u_l(t+t_l,\theta)= \nabla \hat{u}_l(t,\theta)$ converges to $\nabla u(t,\theta)$ on near $\set{0}\times S^1$ smoothly but
\[
	\sup_{\set{0}\times S^1} \abs{\partial_t u + m}=0.
\]

\end{proof}

\section{Applications}
\label{sec:app}
In this section, we derive two applications of the decay estimates that were proved in the previous sections.

\subsection{Removability of singularities}
\label{subsec:remove}

The problem of removability of isolated singularities for Willmore surfaces was first studied by Kuwert and Sch\"atzle for the codimension one case. In their works \cite{K-S2,K-S3}, they demonstrated that under the assumption of finite $L^2$ integral of the second fundamental form and finite area, such surfaces are $C^{1,\alpha}$ smooth. Moreover, the surface can be extended smoothly to be an embedded surface when the singularity residue is zero and the area density is less than 2. Later, Rivi\`ere  proved that such results are valid for higher codimensions \cite{R}. Unlike harmonic mappings, $C^{1,\alpha}$ (for any $\alpha\in (0,1)$) is the best we can hope for isolated singularities of Willmore surfaces without additional conditions (e.g., the inversion of a catenoid). In \cite{B-R0}, Bernard and Rivi\`ere explored whether Willmore surfaces with isolated singularities could be represented as the image sets of smooth mappings.

Based on the discussions in \cite{K-S2,K-S3}, to prove the removability of singularities, we only need to demonstrate the decay of the second fundamental form.

\begin{cor}\label{cor:rem}
	Let $f:D\setminus \set{0}\rightarrow\R^n$ be a conformal and Willmore immersion. If 
\[
	\norm{A_f}_{L^2(D\setminus \set{0})}<\infty \qquad \text{and} \qquad \mbox{Area}(f)<\infty ,
\]
then for any $q\in(0,2)$, we have
\begin{equation}
	\label{eqn:remove}
\int_{(t,t+1)\times S^1}|A_f|^2dV_g\leq C(q,f)e^{-qt} \qquad \text{for} \quad t>0.
\end{equation}
\end{cor}
This corollary will imply that $f$ can be extended to a $C^{1,\alpha}$ map from $D$ into $\R^n$. Hence, we have reproved a result due to Kuwert and Sch\"atzle for the case of $n=3$ \cite{K-S2,K-S3} and  Rivi\`ere for the general case \cite{R}.

Ideally, a possible approach of this corollary is to apply Theorem \ref{thm:out} to the restriction of $f$ on $[0,k]\times S^1$ and take the limit $k\to \infty $. Unfortunately, it is not possible to verify the assumption $\tau_1(f,c)=0$ for all $c\in \Real^n$. Indeed, for the inverted catenoid, this is not true, but it satisfies the assumptions of Corollary \ref{cor:rem}. The only place that we need $\tau_1(f,c)=0$ is in the proof of Theorem \ref{3-circle.H} (and hence Corollary \ref{pro:decayH} for the decay of the mean curvature). 

Here the solution is to prove the decay of the mean curvature directly without appealing to Theorem \ref{3-circle.H}. This is possible because we work with a special solution $f$ defined on $D\setminus \set{0}$. If we run the argument of Theorem \ref{3-circle.H} for a sequence of scaling of $f$, thanks to the known result of Kuwert and Li \cite{K-L}, we do not need any translation in the proof, which spares us the need of having $\tau_1(f,c)=0$ (See Remark \ref{rem:tau1}.)

\proof 
Using cylinder coordinate, we regard $f$ as defined on $[0,\infty )\times S^1$. Fix any $q'\in (q,2)$ and fix any $L>L_0$, where $L_0$ is given in Theorem \ref{3-circle.H}. Recall that $[0,\infty ) \times S^1$ is the union of $Q_i$ for $i=1,2,\cdots $.

{\bf Claim.} There is some $i_0>0$ such that for any $i>i_0$, we have
\begin{equation}
	\label{eqn:3ch}
	\int_{Q_i} \abs{H}^2dV_g\leq e^{-q'L}\left( \int_{Q_{i-1}}\abs{H}^2 dV_g + \int_{Q_{i+1}} \abs{H}^2 dV_g \right).
\end{equation}
Assuming the claim, it follows from the discussion at the beginning of Section \ref{sec:3will} and the finiteness of the integral $\int_{[0,\infty )\times S^1} \abs{H}^2 dV_g$ that we have the decay estimate
\[
	\int_{Q_i} \abs{H}^2 dV_g \leq C e^{-iqL} \qquad \text{for} \quad i>i_0.
\]
The proofs in Section \ref{sec:A} can be repeated line by line for $f_k= f|_{[i_0,i_0+k] \times S^1}$ to get the dacay estimate \eqref{eqn:remove} for $f$.

The rest of the proof is devoted to the proof of the claim. If otherwise, we find a subsequence $i_k$ such that \eqref{eqn:3ch} is not true for $i_k$.

By Theorem 3.1 of \cite{K-L}, we know $f\in W^{2,2}(D)$ and there exists $m\in \mathbb Z_+$ and $u\in C^0(D)$ such that 
\[
	f^*(g_{\Real^n})= e^{2u}\abs{z}^{2(m-1)}(dx^2+dy^2).
\]
Without loss of generality, we assume $f(0)=0$ and $u(0)=0$. It is also proved there that for any sequence $r_k\to 0$, if we set $f_k(z)=\frac{f(r_kz)}{r_k^{m}}$, then $|\nabla f_k|$ is bounded, and hence $f_k$ sub-converges to $\frac{z^m}{m}$ in $C^0_{loc}(D)\cap C^\infty_{loc}(D\setminus\{0\})$ after a rotation that is independent of $k$.

In terms of $(t,\theta)$ coordiante, the above result means that for $t_k=i_k L$,
\[
\hat{f}_k(t,\theta):=e^{mt_k}f(t_k+t,\theta)
\]
converges (after a rotation independent of $k$) to
\[
	f_\infty(t,\theta)=\frac{e^{-mt}}{m}(\cos m\theta,\sin m\theta,0,\cdots,0)
\]
locally smoothly on $\Real \times S^1$. There is a similar part in the proof of Lemma \ref{lem:b}. 

Since $\tau_2$ is independent of the scaling of $f$, we obtain
\[
	\tau_2(f,S)= \tau_2(\hat{f}_k,S) \to \tau_2(f_\infty ,S)=0 \qquad \text{as} \quad k\to \infty .
\]
In particular, $\tau_2(\hat{f}_k,S)=0$ for any $S\in \so$.

Set $\phi_k=\frac{H_{\hat{f}_k}}{\norm{ H_{\hat{f}_k} e^{-mt}}_{L^2(Q_1)}}$ and regard them as defined on $Q=Q_{0}\cup Q_1\cup Q_2$. By our choice of $i_k$, we have
\begin{equation}
	\label{3.w2}
	\int_{Q_1} \abs{\phi_k}^2 e^{-2mt} dt\theta> e^{-qL} \left( \int_{Q_0} \abs{\phi_k}^2 e^{-2mt} dtd\theta+ \int_{Q_2} \abs{\phi_k}^2 e^{-2mt} dtd\theta \right).
\end{equation}
Now we can obtain a contradiction using the same argument in the proof of Theorem \ref{3-circle.H}. To avoid repetition, we list only a few major points. $\phi_k$ satisfies some linear elliptic equation whose coefficients are bounded, hence the above $L^2$ bound on $Q$ implies local smooth convergence on $Q$ that gives us a harmonic function $\phi$. Being the limit of normal vectors, we may write
\[
	\phi=(0,0,\phi_3,\cdots ,\phi_n).
\]
As a consequence of $\tau_2(\hat{f}_k,S)=0$ for any $S\in \so$, we find that in the expansion of $\phi$, the coefficients before $e^{mt}\cos m\theta$ and $e^{mt}\sin m\theta$ are zero. Hence, by Lemma \ref{3-circle.harmonic}, \eqref{eqn:3C.harmonic} holds for $\phi$. This contradicts the limit of \eqref{3.w2}.
\endproof

\subsection{A gap theorem}
\label{subsec:gap}

In the discussion of the bubble tree problem, the gap theorem is very crucial, because it ensures that only a finite number of bubbles appear. To study the compactness of a sequence of conformal Willmore immersions with bounded Willmore energy and topology, we need to have some version of gap theorem for conformal and Willmore surfaces defined on $\C\setminus\{0\}\rightarrow\R^n$. As an application of our decay estimate, we provide such a theorem.  Other discussions on the gap theorem can be found in \cite{K-S}, \cite{B-R2}. Moreover, in \cite{L-N}, the authors classified all Willmore spheres in $\Real^3$ with no more than $3$ singularities (counting multiplicity), which naturally implies a gap theorem. A generalization of the result from Lamm-Nguyen can be found in \cite{M-R3}.

\begin{cor}\label{cor:gap}
	For any $\beta>0$, there exists a positive number $\tau$ depending on $\beta$. Let $f$ be a  conformal and Willmore immersion from $\R \times S^1$ into $\R^n$. Let $f^*(g_{\Real^n})=e^{2u}(dt^2+d\theta^2)$ and assume
	\begin{equation}
		\label{eqn:beta1}
	\sup_{\Real \times S^1}\abs{\nabla u}<\beta.
	\end{equation}
If $\|A_f\|_{L^2(\Real \times S^1)}<\tau$, then $A_f\equiv 0$ .
\end{cor}

\begin{rem}
	We do not know if this gap $\tau$ can be proved to be independent of $\beta$. It is an interesting question. Nevertheless, it is not a problem when we use it to discuss the compactness of Willmore immersions with bounded Willmore energy and bounded genus, because we can derive such upper bound for $\beta$ once the energy upper bound is fixed. 
\end{rem}
\proof 
Fix $q\in (0,2)$ and let $\delta_1$ be given in Theorem \ref{thm:out}.

If there is no such $\tau$, then we have a sequence of $f_k$ satisfying all the assumptions of Corollary \ref{cor:gap} and
\[
	\lim_{k\to \infty } \norm{A_{f_k}}_{L^2(\Real \times S^1)}\to 0,
\]
but each $A_{f_k}$ is not zero somewhere. This sequence satisfies A.1)-A.3) when restricted to any segments $[A_k,B_k]\times S^1$ with $B_k-A_k\to \infty $. Lemma \ref{lem:b} then implies the existence of $M''$ such that the assumption (b) of Theorem \ref{thm:out} holds for $f_k|_{[-l,l] \times S^1}$ for any $l\in \mathbb N$ as long as $k> M''$. Obviously, if $M''$ is chosen larger, the assumption (c) also holds. Now, fix some $k>M''$, if we can apply Theorem \ref{thm:out} to $f_k|_{[-l,l]\times S^1}$ for any $l\in \mathbb N$, we could derive $A_{f_k}\equiv 0$ and obtain a contradiction.

Threrefore, it remains to prove the assumption (a) of Theorem \ref{thm:out}:
\begin{equation}
	\label{eqn:a}
\tau_{1}(f_k,c)=\tau_2(f_k,S)=0 \qquad \text{for all} \quad  c\in \Real^n,\s S\in \so.
\end{equation}

With $k$ fixed, we omit the subscript for simplicity and write $f$ for $f_k$. Let $u$ and $K$ be the corresponding conformal factor and the Gaussian curvature. If $k$ has been chosen large enough, we may assume
\[
	\abs{\int_{\Real \times S^1} K e^{2u} dtd\theta} \leq \frac{1}{2} \int_{\Real \times S^1} \abs{A}^2 e^{2u}dtd\theta<1/4.
\]
Hence, for any $t'<t''$, it follows from the equation $-\Delta u = K e^{2u}$ that 
\begin{equation}
	\label{eqn:K2u}
\left|\int_{\{t'\}\times S^1}\frac{\partial u}{\partial t} d\theta -\int_{\{t''\}\times S^1}\frac{\partial u}{\partial t} d\theta \right|=\left|\int_{(t_1,t_2)\times S^1}Ke^{2u}dtd\theta\right|<\frac{1}{4}.
\end{equation}
Given $t_j\rightarrow+\infty$, by \eqref{eqn:beta1} and the $\epsilon$-regularity, we can find $\lambda_j$ and $c_j$ such that (after passing to subsequence) $\hat{f}_j(t,\theta):= \lambda_j(f(t-t_j,\theta)-c_j)$ converges smoothly locally to a conformal Willmore immersion $f_\infty$ defined on $\Real \times S^1$ and the second fundamental form of $f_\infty $ is zero. By the same argument as in Lemma \ref{lem:b}, we find that if $u_\infty $ and $\hat{u}_j$ are the conformal factors of $f_\infty $ and $\hat{f}_j$ respectively, then
\[
	\partial_t u_\infty = m \qquad \text{for some} \quad m\in \mathbb Z\setminus \set{0}.
\]
Hence,
$$
\frac{1}{2\pi}\int_0^{2\pi}\frac{\partial u}{\partial t}(t_j,\theta) d\theta= \frac{1}{2\pi}\int_0^{2\pi}\frac{\partial \hat{u}_j}{\partial t}(0,\theta) d\theta \rightarrow m.
$$ 
By \eqref{eqn:K2u}, we have 
$$
2\pi m-\frac{1}{4}\leq\frac{d}{dt}\int_{0}^{2\pi}u(t,\theta) d\theta\leq 2\pi m+\frac{1}{4},\qquad \text{for all} \quad t\in \Real.
$$
By a scaling of $f$ if necessary, we may assume $\int_0^{2\pi} u(0,\theta)d\theta=0$, which implies that
$$
(2\pi m-\frac{1}{4})t\leq \int_0^{2\pi}u(t,\theta) d\theta \leq (2\pi m+\frac{1}{4})t, \qquad \text{for all} \quad t\in \Real.
$$
Together with our assumption that $|\nabla u|<\beta$, we get  
$$
-\pi\beta +( m-\frac{1}{8\pi})t\leq u\leq \pi \beta+(m+\frac{1}{8\pi})t.
$$ 
Depending on the sign of $m$, either $Area(f|_{(0,+\infty)\times S^1})$ or $Area(f|_{(-\infty,0)\times S^1})$ is infinity. Assume without loss of generality that $Area(f|_{(0,+\infty)\times S^1})=\infty$ and by the lower bound of $u$, the surface is complete at this end. Combining Theorem 4.2.1 in \cite{M-S} and the $\epsilon$-regularity of \cite{K-S,R}, we know that for some $\tilde{t}_l\to \infty$ and $\tilde{\lambda}_l$ uniformly comparable with $e^{-m \tilde{t}_l}$, the sequence
\[
	\tilde{f}_l(t,\theta):= \tilde{\lambda}_l (f(t+\tilde{t}_l,\theta)-f(\tilde{t}_l,0)) 
\]
converges smoothly locally on $\Real \times S^1$ to a rotation of
\[
	f_\infty (t,\theta)=\frac{ e^{mt}}{m}(\cos m\theta,\sin m\theta,0,\cdots ,0).
\]
Hence, for any $c\in \Real^n$,
\[
	\tau_1(\tilde{f}_l,c)= \tilde{\lambda}_l^{-1} \tau_1(f,c) \to \tau_1(f_\infty ,c)=0 \qquad \text{as} \quad l\to \infty 
\]
which implies that $\tau_1(f,c)=0$. Similarly,
\[
	\tau_2(f,S)=\tau_2(\tilde{f}_l,S)\to \tau_2(f_\infty ,S)=0 \qquad \text{as} \quad l\to \infty 
\]
for any $S\in \so$, which implies that $\tau_2(f,S)=0$.
Hence, we have finished the proof of \eqref{eqn:a} and therefore the whole proof of Corollary \ref{cor:gap}.
%
%
%

\endproof

%
%
\appendix

\section{A 3-circle lemma for harmonic function}
In this section, we prove a 3-circle lemma for the Dirichlet energy of harmonic functions. Such a result is well-known. There are many versions of such results in the literature, but we haven't found a proof for the following version that is needed in this paper. As a service to the readers, we provide the statement and a proof here.

\begin{lem}
	\label{lem:appendix} 
	For any $q\in (0,2)$, there exists $L_0(q)$ such that if $L>L_0$, the following holds. Define $Q_i:=[(i-1)L,iL] \times S^1$.
Let $v$ be a harmonic function from $Q=Q_1\cup Q_2\cup Q_3$. If
\begin{equation}\label{Po.equation}
\int_{Q_2}\left( |\partial_t v|^2- \abs{\partial_\theta v}^2\right) dtd\theta=0,
\end{equation}
then 
\[
\int_{Q_2} \abs{\nabla v}^2 dtd\theta\leq e^{-qL}\left( \int_{Q_1} \abs{\nabla v}^2 dtd\theta + \int_{Q_3} \abs{\nabla v}^2 dtd\theta \right).
\]
\end{lem}

\proof 
Multiplying both sides of $(\partial_t^2 +\partial_\theta^2) v=0$ by $\partial_t v$ and integrating by parts, we get
\[
	\partial_t \int_{\set{t}\times S^1} \left( |\partial_t v|^2- \abs{\partial_\theta v}^2\right) d\theta =0 \qquad \text{for all} \quad t\in [0,3L].
\]
Hence, \eqref{Po.equation} implies that
\begin{equation}
	\label{eqn:goodt}
	\int_{\set{t}\times S^1} \left( |\partial_t v|^2- \abs{\partial_\theta v}^2\right) d\theta =0 \qquad \text{for all} \quad t\in [0,3L].
\end{equation}
By \eqref{eqn:goodt} and integrating over $t$, it suffices to prove
\[
	\int_{\set{t}\times S^1} \abs{\partial_\theta v}^2 d\theta \leq e^{-qL} \left( \int_{\set{t-L}\times S} \abs{\partial_\theta v}^2 d\theta + \int_{\set{t+L}\times S^1} \abs{\partial_\theta v}^2 d\theta \right).
\]
By translation invariance, this is reduced to the special case $t=0$,
\begin{equation}
	\label{eqn:reduce}
	\int_{\set{0}\times S^1} \abs{\partial_\theta v}^2 d\theta \leq e^{-qL} \left( \int_{\set{-L} \times S^1} \abs{\partial_\theta v}^2 d\theta + \int_{\set{L}\times  S^1} \abs{\partial_\theta v}^2 d\theta \right).
\end{equation}
To see this, we notice that $\partial_\theta v$ is a special harmonic function in the sense that in its expansion, there is no constant term or linear term, that is,
\[
\partial_\theta v =\sum_{k=1}^\infty(a_ke^{kt}+b_ke^{-kt})\cos k\theta+\sum_{k=1}^\infty(a_k'e^{kt}+b_k'e^{-kt})\sin k\theta,
\]
By Parseval's theorem, proving \eqref{eqn:reduce} amounts to verify that for any $a,b\in \Real$, we have
\[
	(a+b)^2 \leq e^{-qL}\left( (ae^L+be^{-L})^2 + (ae^{-L}+be^{L})^2 \right)
\]
as long as $L>L_0$. The existence of such $L_0$ is elementary.

%

\endproof

{\small

\begin{thebibliography}{99}

\bibitem{B} Y. Bernard: Noether's theorem and the Willmore functional. {\em Adv. Calc. Var.} {\bf 9} (2016), no. 3, 217-234.
		
		
\bibitem{B-R0}Y. Bernard, T.Rivi\`ere: Singularity removability at branch points for Willmore surfaces.
{\em Pacific J. Math.} {\bf 265} (2013), no. 2, 257-311.		

\bibitem{B-R} Y. Bernard, T.Rivi\`ere: Energy quantization for Willmore surfaces and applications. {\em Ann. of Math.} (2) {\bf 180} (2014), no. 1, 87-136.

\bibitem{B-R2} Y. Bernard, T.Rivi\`ere: An Energy Gap Phenomenon for Willmore Spheres. https://people.math.ethz.ch/~triviere/pdf/pub/willmore-gap.pdf

\bibitem{C-L} J. Chen, Y. Li: Bubble tree of branched conformal immersions and applications to the Willmore functional. {\em Amer. J. Math.} {\bf 136} (2014), no. 4, 1107-1154.

\bibitem{C-T}J. Chen, G. Tian: Compactification of moduli space of harmonic mappings. {\em
Comment. Math. Helv.} {\bf 74} (1999), no. 2, 201-237.

\bibitem{C-L-W} L. Chen, Y. Li, Y. Wang: The refined analysis on the convergence behavior of harmonic map sequence from cylinders. {\em J. Geom. Anal.} {bf 22} (2012), no.4, 942-963.

\bibitem{C-D-M} T. Colding, C. De Lellis, W. Minicozzi: Three circles theorems for Schr\"odiner operators on cylindrical ends and geometric applications. {\em Comm. Pure Appl. Math.} {\bf 61} (2008), no. 11, 1540-1602.

\bibitem{D} W. Ding: manuscript.

\bibitem{D-T} W. Ding, G. Tian: Energy identity for a class of approximate harmonic maps from surfaces. {\em Comm. Anal. Geom.} {\bf 3} (1995), no.3-4, 543-554.



\bibitem{K-L} E. Kuwert, Y. Li: $W^{2,2}-$conformal immersions of a closed Riemann surface into $\Real^n$,
{\em Comm. Anal. Geom.} {\bf 20} (2012), no. 2, 313-340.

\bibitem{K-S}
E. Kuwert, R. Sch\"atzle: The Willmore flow with small initial energy
{\em J. Differential Geom.} {\bf 57} (2001), no. 3, 409-441.

\bibitem{K-S2} E. Kuwert, R. Sch\"atzle: Removability of point singularities of Willmore surfaces.(English summary)
{\em Ann. of Math. }(2){\bf 160} (2004), no.1, 315-357.

\bibitem{K-S3} E. Kuwert, R. Sch\"atzle:  
Branch points of Willmore surfaces.
{\em Duke Math. J. }{\bf 138} (2007), no. 2, 179-201.

\bibitem{L-N}T. Lamm, H. Nguyen: Branched Willmore spheres.
{\em J. Reine Angew. Math.} {\bf 701} (2015), 169-194.

\bibitem{L-R}P.Laurain, T. Rivi\`ere: Energy quantization of Willmore surfaces at the boundary of the moduli space.
{\em Duke Math. J.} {\bf 167} (2018), no. 11, 2073-2124.

\bibitem{L-W-Z}Y. Li, G. Wei, Z. Zhou: John-Nirenberg radius and collapsing in conformal geometry. {\em Asian J. Math.} {\bf 24} (2020), no. 5, 759-782.

\bibitem{L-W} F.-H. Lin, C.-Y. Wang: Energy idenity of harmonic map flows from surfaces at finite singular time. {\em Calc. Var.} {\bf 6} (1998), 369-380.

\bibitem{L} G. Liu: Three-circle theorem and dimension estimate for holomorphic functions on K\"ahler manifolds. {\em Duke Math. J.} {\bf 165} (2016), no. 15, 2899-2919.

\bibitem{M} D. Martino: Energy quantization for Willmore surfaces with bounded index. {\em arXiv:2305.08668}.

\bibitem{M-R3} A. Michelat, T. Rivi\`ere: 
The classification of branched Willmore spheres in the 3-sphere and the 4-sphere. {\em Ann. Sci. \'Ec. Norm. Sup\'er. (4)} {\bf 55} (2022), no.5, 1199-1288.

\bibitem{M-R2} A. Michelat, T. Rivi\`ere: Pointwise expansion of degenerating immersions of finite total curvature. {\em J. Geom. Anal.} {\bf 33}, (2023).

\bibitem{M-R1} A. Michelat, T. Rivi\`ere: Morse Index stability of Willmore immersions I. {\em arXiv:2306.04608}.



\bibitem{M-S} S. M\"uller and V. \v{S}ver\'ak: On surfaces
of finite total curvature, {\em J. Differential Geom.}
{\bf 42} (1995),
229-258.


\bibitem{P} T. H. Parker: Bubble tree convergence for harmonic maps
, {\em J. Differential Geom.} {\bf 44} (1996), no. 3, 595-633.

\bibitem{Q-T}J. Qing, G. Tian: Bubbling of the heat flows for harmonic maps from surfaces.
{\em Comm. Pure Appl. Math.} {\bf 50} (1997), no. 4, 295-310.

\bibitem{R}T. Rivi\`ere: Analysis aspects of Willmore surfaces
{\em Invent. Math.} {\bf 174} (2008), no. 1, 1-45.

\bibitem{R-V}E.A. Ruh, J. Vilms: The tension field of the Gauss map.
{\em Trans. Amer. Math. Soc.} {\bf 149} (1970), 569-573.

\bibitem{S-U} J. Sacks, K. Unlenbeck: The existence of minimal immersions of  2-spheres
Sacks, {\em Ann. of Math.} (2) {\bf 113} (1981), no. 1, 1-24.

\bibitem{W-W-Z}W. Wang, D. Wei, Z. Zhang: Energy identity for approximate harmonic maps from surfaces to general targets. {\em J. Funct. Anal.} {\bf 272} (2017), no.2, 776-803.

\bibitem{X} G. Xu: Three circles theorems for harmonic functions. {\em Math. Ann.} {\bf 366} (2016), no. 3-4, 1281-1317.
\end{thebibliography}
\end{document}